\documentclass[a4paper,12pt]{article}

\usepackage{amsmath}
\usepackage{amssymb}
\usepackage{amsfonts}
\usepackage{mathrsfs} 
\usepackage{graphicx}
\usepackage{color}
\usepackage{ulem}
\usepackage{bm}

\newtheorem{proposition}{Proposition}

\newtheorem{definition}{Definition}
\newtheorem{assumption}{Assumption}
\newtheorem{remark}{Remark}

\newenvironment{proof}{\noindent {\bf Proof:}}{\hfill $\Box$}

\newcommand{\R}{\mathbb{R}} 
\newcommand{\N}{\mathbb{N}} 

\newcommand{\an}[1]{{\color{black}  #1}}

\newcommand\Set[2]{\left\{\,#1\mbox{ }\middle|\mbox{ }#2\,\right\}}
\newcommand{\defeq}{:=}

\title{\bf Solving unbounded calculus of variations problems with the moment-SOS hierarchy\footnote{This work was funded by the European Union under the project ROBOPROX
(reg. no. CZ.02.01.01/00/22 008/0004590). \an{It was also supported by the Grant Agency of the Czech Technical University in Prague (grant no.
SGS25/145/OHK3/3T/13) and National Research Foundation, Prime Minister’s Office, Singapore, under its Campus for Research Excellence and Technological Enterprise (CREATE) program.}}}

\begin{document}

\author{Karol\'{\i}na Sehnalov\'a$^1$, Didier Henrion$^{1,2}$, Milan Korda$^{1,2}$, Martin Kru\v z\'{\i}k$^{3,4}$}

\footnotetext[1]{Faculty of Electrical Engineering, Czech Technical University in Prague, Technick\'a 2, CZ-16627 Prague, Czechia.}
\footnotetext[2]{CNRS; LAAS; Universit\'e de Toulouse, 7 avenue du colonel Roche, F-31400 Toulouse, France. }
\footnotetext[3]{Institute of Information Sciences and Automation, Czech Academy of Sciences, Pod Vod\'arenskou v\v e\v z\'{\i}~4, CZ-18200 Prague, Czechia.}
\footnotetext[4]{Faculty of Civil Engineering, Czech Technical University in Prague, Th\'{a}kurova 7, CZ-16627 Prague, Czechia.}

\date{Draft of \today}

\maketitle
\begin{abstract}
The behaviour of the moment-sums-of-squares (moment-SOS) hierarchy for polynomial optimal control problems on compact sets has been explored to a large extent. Our contribution focuses on the case of non-compact control sets. We describe a new approach to optimal control problems with unbounded controls, using compactification by partial homogenization, leading to an equivalent infinite dimensional linear program with compactly supported measures. Our results are closely related to the results of a previous approach using DiPerna-Majda measures. However, our work provides a sound proof of the absence of relaxation gap, which was conjectured in the previous work, and thereby enables the design of a moment-sum-of-squares relaxation with guaranteed convergence.
\end{abstract}

\section{Introduction}
\label{sec:introduction}
The moment-sums-of-squares (moment-SOS) hierarchy is a powerful tool that can be used to solve various non-linear non-convex problems to global optimality, at the price of solving a collection of convex optimization problems of increasing size, see e.g. \cite{lasserre_positive_2010,korda_moment_2020,nie_moment_2023} and references therein. However, the proof of convergence of the hierarchy relies on the assumption that the variables are restricted to a compact set. In particular, this is the case for the optimal control problems with polynomial data solved with the moment-SOS hierarchy in \cite{lasserre_nonlinear_2008} and \cite{korda_moment_2020}. In these references, the control  (as well as the time and state) variables are restricted to a given compact set.

In the context of polynomial optimization problems (POPs) without differential equations and control, compactness assumptions can be sometimes relaxed. For example, analyticity or growth conditions can be enforced on the moment sequences (using the so-called Carleman and Nussbaum conditions, see e.g. \cite[Prop. 3.5 and Section 3.4.2]{lasserre_positive_2010}) for solving unbounded POPs. Homogenization can be used to reformulate unbounded POP as bounded POP with the use of projective coordinates, as explained e.g. in \cite{huang_homogenization_2022} or \cite[Section 8.5]{nie_moment_2023}.

Our work is dedicated to optimal control problems for which the control variable can be unbounded, i.e. it can attain infinite values. In the context of the moment-SOS hierarchy, such problems were analyzed in \cite{claeys_semi-definite_2014, henrion_optimal_2019} with the use of DiPerna-Majda measures, a tool from functional analysis and partial differential equations. Other approaches suggest, for example, increasing bounds on the control and then using the moment-SOS hierarchy, as described in \cite{fantuzzi_global_2024}.
There exist versions of the maximum principle for unbounded controls, see e.g. \cite{clarke_control}, but the interplay between these first order local optimality conditions and the globally optimal bounds provided by the moment-SOS hierarchy remains to be explored.

We propose a series of linear programming (LP) reformulations of the unbounded optimal control problem, arriving at a reformulation that is suitable for the numerical solution via the moment-SOS hierarchy. We believe our approach is more direct than previous work, as it bypasses DiPerna-Majda measures, and it relies on a simpler partial homogenization technique similar to what was originally developed for POP. Since we are dealing with differential equations, our reformulations revolve around occupation measures, which are used to describe trajectories. These occupation measures are supported on non-compact sets, which are then compactified using homogenization. This gives rise to a reformulation that contains rational terms. The last reformulation is designed to restore polynomial data, so that it can be solved via the moment-SOS hierarchy. More precisely, a primal formulation on (pseudo-) moments of positive measures is proposed and examined, rather than a dual on (SOS) positive polynomials. However, after reformulations, our work and \cite{claeys_semi-definite_2014,henrion_optimal_2019} propose the same LP on measures eligible for the moment-SOS hierarchy, even though the approaches are different. Hence, our work can be seen as a simplification of \cite{claeys_semi-definite_2014,henrion_optimal_2019}. Furthermore, we rigorously prove that all these LP reformulations achieve the same optimal value as the optimal control problem with relaxed controls, i.e. without relaxation gap, which has not been done in the previous work.

\subsection*{Contribution}

In summary, for solving unbounded optimal control problems with the moment-SOS hierarchy, the main contributions of this work are twofold:
\begin{itemize}
    \item simplification of previous approaches;
    \item rigorous proof of no relaxation gap.
\end{itemize}

\subsection*{Outline}

Our optimal control problem with unbounded control and regularity assumptions is stated in Section \ref{sec:statement}. We then introduce 3 LP occupation measure formulations of the optimal control problem: in Section \ref{sec:lpo} an LP with polynomial data and unbounded controls, in Section \ref{sec:lpc} an LP with rational data and bounded controls, and in Section \ref{sec:lpp} an LP with polynomial data and bounded controls. This latter LP can be then approximately solved with the moment-SOS hierarchy. We prove that these 3 LPs have the same value as the original problem, i.e. there is no relaxation gap. Our approach is then illustrated on 2 numerical examples in Section \ref{sec:examples}.

\subsection*{Notations}

By measures we understand real valued countably additive functions on Borel sets of Euclidean space \cite{kolmogorov_introductory_1975}. The support of a measure $\mu$ is denoted $\mathrm{spt}\mu$. Given a set $\mathscr{A}$, let $\mathcal{P}(\mathscr A)$ denote the set of probability measures supported on $\mathscr A$ and let $\mathcal{M}_+(\mathscr A)$ denote the cone of positive (i.e. non-negative) measures supported on $\mathscr A$. The Lebesgue measure is denoted $\lambda$. It is such that $\int f(t)d\lambda(t) = \int f(t)dt$ is the average value of $f$. The Dirac measure at $s$ is denoted $\delta_s$. It is such that $\int f(t)\delta_s(dt) = f(s)$ is the pointwise value of $f$. The space of continuously differentiable functions $f: \mathscr A \to \R$ is denoted $C^1(\mathscr A)$. The time derivative of function $f(t)$ is denoted $\dot{f}(t)$. Given a set $\mathscr B$, let $W^{1,p}(\mathscr A; \mathscr B)$ denote the Sobolev space of functions $f: \mathscr A \to \mathscr B$ whose weak derivative belongs to $L^p({\mathscr A})$, the Lebesgue space with exponent $p$ \cite{Dacorogna_2008}.

\section{Problem statement}\label{sec:statement}

We consider the following class of optimal control problems:
\begin{align}
\label{problem_gradient_diff_constr}
    p^* \defeq &\inf_{x,u} \int_a^b{l\left(t, x(t), u(t)\right)}dt,\\ \notag
    \text{s.t.: } & x(a) = x_a, \quad x(b) = x_b,\notag \\ & x(t) \in \mathscr X, \quad \dot{x}(t) = u(t) \quad \text{a.e. } t \in [a, b], \notag\\ 
    &u\in L^r\left( [a, b]\right) \notag  
\end{align} 
where $l$ is a polynomial, $\mathscr X$ is a compact interval, $p\in \N$, $p \geq 1$, and $r$ is the maximum of 1 and the degree of $l$ in $u$. Note that this setup allows the control $u(t)$ to be unbounded. Since $u(t) \in L^r([a,b])$, then $x(t) \in W^{1,r}([a,b]; {\mathscr X})$.

\begin{assumption}\label{ass:coercivity}
There exists a minimizing sequence $(u_k)_{k=1}^\infty \subset L^r([a,b])$  satisfying the bound $\int_a^b |u_k(t)|^r dt < C$ for some $C < \infty$ .
\end{assumption}
\an{Note that enforcing a bound on the Lebesgue norm of the control does not imply that the control is bounded pointwise. Also, under this assumption, the objective value is finite for any element of the minimizing sequence.}
\begin{remark}
\label{remark_coercivity}
    This assumption is satisfied for example when $l$ is nonnegative and coercive with respect to the $u$ variable (i.e, for every positive constant $C_1$ there exists a positive constant $C_2$ such that $l(t,x,u)\geq C_2|u|^r$ whenever $|u| \geq C_1$). \an{Concretely, we can set $C = (b-a)C_1^r + k/C_2$, where $(C_1,C_2)$ is any pair from the definition of coercivity and $k \ge p^*$. }
\end{remark}

We consider the scalar case just for the ease of exposition, the developed technique can be extended easily to multivariate case and $\mathscr X$ would be a general compact semialgebraic set \cite[Chapter 1]{lasserre_positive_2010} instead of an interval. This class of problems is a particular case of optimal control problems called (scalar) calculus of variations problems, since we control the whole gradient, which is here the velocity.

\section{Linear formulation using occupation measures}\label{sec:lpo}

First, let us define relaxed controls (a.k.a. Young measures) and occupation measures \cite{young_lectures_1969,lewis_relaxation_1980,vinter_1993}.
\begin{definition}[Relaxed control]
Given an integer $r \geq 1$, a relaxed control is a measurable function $t \in [a,b] \mapsto \omega_t \in \mathcal P(\R)$, such that there exists a sequence $(u_k)_{k=1}^\infty \subset L^r([a,b])$ satisfying
\begin{align*}
    \lim_{k\to\infty}\int_a^bv(u_k(t))dt = \int_a^b\int_{\R}v(u)d\omega_t(u)dt
\end{align*}
for every $v\in L^1([a, b])$. We denote the set of relaxed controls by $\mathcal U_{\text{r}}([a, b]; \mathcal P(\R))$.
\end{definition}
Relaxed controls are just time-dependent probability measures. We can now relax problem (\ref{problem_gradient_diff_constr}) with classical controls as a problem with relaxed controls:
\begin{align}
\label{problem_relaxed}
    p_{\text{r}}^* \defeq &\inf_{x,\omega} \int_a^b{\int_{\R}l\left(t, x(t), u\right)\omega_t(du)}dt,\\ \notag
    \text{s.t.: } & x(a) = x_a, \quad x(b) = x_b,\notag \\ &x(t) \in \mathscr X, \quad
    \dot{x}(t) = \int_{\R}u\omega_t(du) \quad \text{a.e. } t \in [a, b],\notag \\
    & 
    \omega \in \mathcal U_{\text{r}}([a, b]; \mathcal P(\R)). \notag
\end{align}
Note that $ p_{\text{r}}^* \leq p^*$, because in general, there are more relaxed controls than those corresponding to classical controls, which correspond to the particular choice $\omega_t(du)=\delta_{u(t)}$.
We say that a trajectory $x(t)$ is admissible, provided that it satisfies the constraints of \eqref{problem_relaxed}.

\begin{definition}[Occupation measure]
We define the occupation measure associated to an admissible trajectory $x(t)$ as
\begin{equation*}
    d\mu\left(t,x,u\right) \defeq \lambda(dt) \delta_{x(t)}(dx) \omega_t(du).
\end{equation*}
The set of all occupation measures of admissible trajectories is denoted $\mathcal{O}(\Omega)$, where \an{$\Omega \defeq [a, b] \times \mathscr X \times \R$.}
\end{definition}
Note that whenever $\omega_t(du) = \delta_{u(t)}$, the occupation measure satisfies the following relation
\begin{align*}
    &\int_{\Omega}v(t, x, u)d \mu(t, x, u) = \int_a^bv(t, x(t), u(t))dt
\end{align*}
for every test function $v \in L^1(\Omega)$,
which indicates that integration with respect to an occupation measure corresponds to time averaging along its trajectory.

Now consider a test function $v \in C^1\left( [a, b] \times \mathscr X \right)$. Its time average over a trajectory $x(t)$ can be expressed as
\begin{align*}
    \int_a^bdv\left( t, x\left(t\right) \right) &= v\left(b, x_b \right) - v\left(a, x_a \right)\\ &= \int_{[a, b]\times \mathscr X}v \delta_b\delta_{x_b} - \int_{[a, b]\times \mathscr X}v \delta_a\delta_{x_a}
\end{align*}
or 
\begin{align*}
    &\int_a^bdv\left( t, x\left(t\right) \right) = \int_a^b\left(\frac{\partial v}{\partial t} + \frac{\partial v}{\partial x}u \right)\left( t, x(t), u(t)\right)dt \\ =& \int_{\Omega}\left(\frac{\partial v}{\partial t} + \frac{\partial v}{\partial x}u \right)\left( t, x, u\right)d\mu =  -\int_{\Omega}v\left( \frac{\partial \mu}{\partial t} + \frac{\partial \mu}{\partial x}u \right).
\end{align*}
Consequently, for every test function $v$ it holds
\begin{equation*}
    \int_{\Omega}v\left( \frac{\partial \mu}{\partial t} + \frac{\partial \mu}{\partial x}u \right) + \int_{[a, b]\times \mathscr X}v \delta_b\delta_{x_b} = \int_{[a, b]\times \mathscr X}v \delta_a\delta_{x_a},
\end{equation*}
which we can write in shorter notation as
\begin{equation}
\label{gradient_liouville_original}
    \frac{\partial \mu}{\partial t} + \frac{\partial \mu}{\partial x}u + \delta_b\delta_{x_b} = \delta_a\delta_{x_a}.
\end{equation}
Equation (\ref{gradient_liouville_original}) models the differential constraint $\dot{x}(t) = u(t)$ together with the boundary conditions and we will refer to it as the Liouville equation, see e.g. \cite{korda_convex_2014} and references therein. However, the Liouville equation does not take into account that we want to stay in the space of admissible trajectories. This means that there may be occupation measures satisfying the Liouville equation and corresponding to trajectories which do not belong to $W^{1, r}([a, b]; \mathscr X)$. To restrict the solutions of the Liouville equation accordingly, we can add the constraint $\int_{\Omega}|u|^rd\mu\leq C$ without changing the infimum in view of Assumption \ref{ass:coercivity}.

\begin{remark}
Note however that enforcing $\int_a^b |u(t)|^r dt \leq C$ explicitly as a constraint in the original problem \eqref{problem_gradient_diff_constr} may introduce a relaxation gap, as explained in \cite[Section 4.1]{korda_rios_2022}.
\end{remark}

\begin{proposition}[Superposition of trajectories]
\label{prop_decomposition1}
    Let $\mu \in \mathcal{M}_+\left( \Omega \right)$ satisfy the Liouville equation \eqref{gradient_liouville_original}  and the constraint $\int_a^b |u|^r d\mu \leq C$. Then it can be decomposed into a family of occupation measures $\mu = \int_{\mathcal{O}(\Omega)}{\gamma}d\eta(\gamma)$ for some $\eta \in \mathcal P(\mathcal{O}(\Omega))$. This family of occupation measures corresponds to a family of admissible trajectories.
\end{proposition}
\begin{proof}
    We use \cite[Theorem 5.2]{bernard_young_2008}. We say that the measure $\mu$ is a transport measure in the sense described in \cite{bernard_young_2008}. For that, we need to check that
    \begin{equation*}
        \int_{[a, b]\times \R \times \R}{\left(\frac{\partial v}{\partial t} + \frac{\partial v}{\partial x}u\right)}d\mu(t, x, u) = 0
    \end{equation*}
    for all smooth compactly supported functions $v: (a, b) \times \R \to \R$, i.e. such that $v(a,x)=v(b,x)=0$ for all $x \in \R$. \an{Although we do not assume test functions vanishing at the boundary in the derivation of \eqref{gradient_liouville_original}, here we do so to fit the framework of \cite{bernard_young_2008}.} Since $\mu$ satisfies the Liouville equation and $\textrm{spt}\mu \subseteq \Omega$, we can write
    \begin{align*}
        &\int_{[a, b]\times \R \times \R}{\left(\frac{\partial v}{\partial t} + \frac{\partial v}{\partial x}u\right)}d\mu(t, x, u)\\ =& \int_{\Omega}{\left(\frac{\partial v}{\partial t} + \frac{\partial v}{\partial x}u\right)}d\mu(t, x, u) = v(b, x_b) - v(a, x_a) = 0.
    \end{align*}
    Now we can apply \cite[Theorem 5.2]{bernard_young_2008}, which says, that there exists a Borel probability measure $\eta$ on the set of generalized curves. From the constraint $\int_a^b |u|^r d\mu \leq C$, we know that these are generated by sequences $(u_k)_{k=1}^\infty \subset L^r([a,b])$, meaning that these generalized curves are occupation measures of $\mathcal O([a, b]\times \R \times \R)$, i.e. $\mu = \int_{\mathcal{O}([a, b]\times \R \times \R])}{\gamma}d\eta(\gamma)$. Hence, $\eta$ is a decomposition of $\mu$ \an{into} occupation measures, which correspond to continuous curves $x(t): [a, b] \to \R$. Since $\textrm{spt}\mu \subset \Omega = [a, b] \times \mathscr X \times \R$, we can take $\eta$ such that $\textrm{spt}\eta \subset \mathcal{O}(\Omega)$. Hence, $\mu$ can be decomposed \an{into} a family of occupation measures, which corresponds to a family of continuous curves $x(t): [a, b] \to \mathscr X$.
\end{proof}

Now we can formulate an LP on measures
\begin{align}
\label{problem_ocm}
    p_{\text{m}}^* \defeq &\inf_{\mu \in \mathcal{M}_+\left( \Omega \right)} \int_\Omega{l\left(t, x, u\right)}d\mu,\\
    \text{s.t.: } &\frac{\partial \mu}{\partial t} + \frac{\partial \mu}{\partial x}u + \delta_b\delta_{x_b} = \delta_a\delta_{x_a}, \quad \int_{\Omega}|u|^rd\mu \leq C \notag,
\end{align}
where the last constraint is added without changing the infimum in view of Assumption~\ref{ass:coercivity}, as discussed after the derivation of the Liouville equation. \an{This constraint will play an important role later on, ensuring the convergence of the hierarchy}. Notice that we optimize over all measures from the positive cone $\mathcal{M}_+\left(\Omega \right)$ instead of solely the set of occupation measures corresponding to optimization over admissible trajectories and relaxed controls. To justify this choice, we state the following proposition, which ensures that there is no relaxation gap when we optimize over all measures instead of occupation measures.

\begin{proposition}\label{prop:pm=pb}
    It holds that $p_{\text{m}}^* = p^*_{\text{r}}$.
\end{proposition}
\begin{proof}
    We immediately have $p_{\text{m}}^* \leq p^*_{\text{r}}$, because whenever we have a trajectory $x^*(t)$ and a relaxed control $\omega_t^*(du)$ optimal for problem (\ref{problem_relaxed}), we can choose an occupation measure $d\mu = \lambda(dt)\delta_{x^*(t)}\omega_t^*(du)$, which is admissible for problem (\ref{problem_ocm}) and achieves the same objective value. 
    
    Conversely, let us prove that $p_{\text{m}}^* \geq p^*_{\text{r}}$. Suppose we have a measure $\mu^*$ optimal for the problem (\ref{problem_ocm}). Then $\mu^*$ satisfies the Liouville equation and hence, according to Proposition \ref{prop_decomposition1}, it can be decomposed into a family of occupation measures. More precisely, $\mu^* = \int_{\mathcal{O}(\Omega)}{\gamma}d\eta(\gamma)$ for some $\eta \in \mathcal P\left(\mathcal O\left(\Omega\right)\right)$. Clearly, every $\gamma \in \mathrm{spt}\eta$ is feasible for problem (\ref{problem_ocm}) and therefore, $\int_{\Omega}{l(t, x, u)d\gamma} \geq \int_{\Omega}{l(t, x, u)d\mu^*}$. Furthermore, observe that every $\gamma \in \mathrm{spt}\eta$ is optimal for problem (\ref{problem_ocm}). Indeed, we have
    \begin{align*}
        &\int_{\Omega}{l(t, x, u)d\mu^*} = \int_{\mathcal O(\Omega)}{\int_{\Omega}{l(t,x, u)d\gamma}d\eta(\gamma)} \\ \geq& \int_{\mathcal O(\Omega)}{\int_{\Omega}{l(t,x, u)d\mu^*}d\eta(\gamma)} = \int_{\Omega}{l(t, x, u)d\mu^*}.
    \end{align*}
    In order for the equality to hold, $\int_{\Omega}{l(t,x, u)d\gamma} = \int_{\Omega}{l(t, x, u)d\mu^*}$ for every $\gamma \in \mathrm{spt}\eta$. Thus, we can choose arbitrary $\gamma_1 = \lambda(dt) \delta_{x(t)}(dx) \omega_t(du) \in \mathrm{spt}\eta$, which is optimal for (\ref{problem_ocm}) and take $x(t)$, $\omega_t(du)$ feasible for problem (\ref{problem_relaxed}), which achieves the same objective value.
\end{proof}

\section{Compactification using homogenization}\label{sec:lpc}

In this section, we use a change of variables to deal with the fact that $\Omega$ (the support of positive measure $\mu$ from problem \eqref{problem_ocm}) is unbounded. Recall that both the time interval $[a, b]$ and the state interval $\mathscr X$ are compact. We use partial homogenization with respect to the control variable $u$, which may be unbounded. The concept of homogenization in polynomial optimization is described e.g. in \cite{huang_homogenization_2022} or \cite[Section 8.5]{nie_moment_2023}. Consider the change of coordinates:
\begin{equation}
\label{gradient_coordinate_change}
    u = \frac{z}{w},
\end{equation}
where
\begin{equation}
\label{gradient_coordinate_constraints}
    (z, w) \in \mathscr B_s \defeq \Set{(z, w) \in \R^2}{z^s + w^s = 1, w \geq 0}
\end{equation}
for some natural number $s\geq 1$. Due to the spherical constraint, both $z$ and $w$ are bounded and $\mathscr B_s$ is compact whenever $s$ is even. For simplicity, we will assume that $s$ is even.
\begin{remark}
\label{remark_s_odd}
For $s$ odd, we must proceed with caution and deal with this situation individually. Some cases with $s$ odd will be discussed in the examples section. \an{In general, one may deal with odd $s$ by adding lifting variables as described in \cite[Section 6.3]{claeys_semi-definite_2014}. Alternatively, the control domain could be split into two parts - we can introduce an occupation measure whose control support is positive, and another one whose control support is negative.}
\end{remark}

\begin{proposition}
\label{prop_bijection}
The mapping $g: \R \cup \{\pm \infty\} \to \mathscr B_s,  u \mapsto (z, w) \an{= \left( \frac{u}{\sqrt[s]{1+u^s}}, \frac{1}{\sqrt[s]{1+u^s}} \right)}$ such that $g(\pm \infty)=(\pm1, 0)$ is a bijective homeomorphism.
\end{proposition}

\begin{proof}
The mapping $u \mapsto \left( \frac{u}{\sqrt[s]{1+u^s}}, \frac{1}{\sqrt[s]{1+u^s}} \right)$ clearly satisfies (\ref{gradient_coordinate_change}) - (\ref{gradient_coordinate_constraints}) and for every $z, w$, $w > 0$ we have a pre-image $u$ defined by equation (\ref{gradient_coordinate_change}). The pre-image of $\left(\pm 1, 0 \right)$ is $\pm \infty$ according to our definition. Therefore, $g$ is a surjection. Whenever $g(u_1) = g(u_2)$, that means 
\begin{equation*}
    \left( \frac{u_1}{\sqrt[s]{1+u_1^s}}, \frac{1}{\sqrt[s]{1+u_1^s}} \right) = \left( \frac{u_2}{\sqrt[s]{1+u_2^s}}, \frac{1}{\sqrt[s]{1+u_2^s}} \right),
\end{equation*}
we have $u_1 = u_2$. The equality of the second components implies that $|u_1| = |u_2|$ and from the equality of the first components, it must hold that $u_1 = u_2$. This concludes that $g$ is an injection. Since $g$ is continuous and has an inverse, $g^{-1}: (z, w)\mapsto \frac{z}{w}$, with $(\pm1, 0) \mapsto \pm \infty$, which is continuous as well, $g$ is a homeomorphism.
\end{proof}

 We can now define the compact set
 \begin{equation*}
 \Omega_{\text{h}} \defeq \Set{\left( t, x, z, w \right) \in \R^4}{t \in [a, b], x \in \mathscr X,(z, w)\in \mathscr B_s }
\end{equation*}
which contains new homogenized variables. Now we can formulate LP (\ref{problem_ocm}) in the homogenized coordinates as follows:
\begin{align}
\label{problem_gradient_lp_rational}
    p_{\text{b}}^* \defeq &\inf_{\mu \in \mathcal{M}_+\left( \Omega_{\text{h}} \right)} \int_{\Omega_{\text{h}}}{l\left(t, x, \frac{z}{w}\right)}d\mu,\\
    \text{s.t.: } &\frac{\partial \mu}{\partial t} + \frac{\partial \mu}{\partial x}\frac{z}{w} + \delta_b\delta_{x_b} = \delta_a\delta_{x_a}, \quad \int_{\Omega}\frac{|z|^r}{w^r}d\mu \leq C. \notag
    \end{align}
Since the mapping $g$ of Proposition \ref{prop_bijection} is a bijection, problems (\ref{problem_ocm}) and (\ref{problem_gradient_lp_rational}) are equivalent, achieving the same optimum ($p_{\text{b}}^* = p_{\text{m}}^*$). Note, however, that it is not achieved by the same optimal measure. When using the substitution (\ref{gradient_coordinate_change}), we introduced rational terms and the data in LP (\ref{problem_gradient_lp_rational}) are not polynomial.

\section{From rational to polynomial data}\label{sec:lpp}

\label{recover_polynomial_property}
We restore polynomial data in the same manner as it is described in \cite[Section 8.5]{nie_moment_2023} for polynomial optimization. Letting $\tilde l(t, x, z, \an{w}) \defeq w^rl(t, x, \frac{z}{w})$, we can rewrite LP (\ref{problem_gradient_lp_rational}) as
\begin{align}
    \label{problem_gradient_lp_rational_denominator}
    p_{\text{b}}^* \defeq &\inf_{\mu \in \mathcal{M}_+\left( \Omega_{\text{h}} \right)} \int_{\Omega_{\text{h}}}{\frac{\Tilde{l}\left(t, x, z, \an{w}\right)}{w^r}}d\mu,\\
    \text{s.t.: } &\frac{\partial \mu}{\partial t} + \frac{\partial \mu}{\partial x}\frac{z}{w} + \delta_b\delta_{x_b} = \delta_a\delta_{x_a}, \quad \int_{\Omega}\frac{|z|^r}{w^r}d\mu \leq C. \notag
\end{align}
From now on, we will assume that $r$ is even, which will allow us to remove the absolute value in the inequality constraint. This corresponds to the assumption that $s$ is even, and similarly, as was suggested in Remark \ref{remark_s_odd}, we will deal with cases with $r$ odd in the examples section. \an{Notice that since $r$ is the maximum of the degree of $l$ in $u$ and 1, $\tilde l$ is a polynomial.} Now consider the following LP with purely polynomial data
\begin{align}
    \label{problem_gradient_lp_rational_polynomial}
    p_{\text{p}}^* \defeq &\inf_{\nu \in \mathcal{M}_+\left( \Omega_{\text{h}} \right)} \int_{\Omega_{\text{h}}}\Tilde{l}\left(t, x, z, \an{w}\right)d\nu,\\
    \text{s.t.: } &\frac{\partial \nu}{\partial t}w^r + \frac{\partial \nu}{\partial x}zw^{r-1} + \delta_b\delta_{x_b} = \delta_a\delta_{x_a}, \quad \int_{\Omega}z^rd\nu \leq C. \notag
\end{align}

\begin{proposition}[No relaxation gap]
\label{prop_substitution_measures_equivalence}
It holds $p_{\text{p}}^* = p_{\text{r}}^*$.
\end{proposition}

\begin{proof}
In view of Proposition~\ref{prop:pm=pb} and the equality $p_{\text{b}}^* = p_{\text{m}}^*$, we have $p_{\text{b}}^* = p_{\text{r}}^*$. Therefore, it suffices to prove $p_{\text{b}}^* = p_{\text{p}}^*$.

    Let measure $\nu$ be feasible for LP (\ref{problem_gradient_lp_rational_polynomial}). Then the measure $\mu = w^{-r}\nu$ is feasible in LP (\ref{problem_gradient_lp_rational_denominator}) and it achieves the same objective value. Hence, $p_{\text{b}}^* \leq p^*_{\text{p}}$. 
    
    Conversely, given a measure $\mu$ feasible
    in (\ref{problem_gradient_lp_rational_denominator}), we define a measure $\nu = w^r\mu$, which is  feasible for LP (\ref{problem_gradient_lp_rational_polynomial}) and thus, we have $p^*_{\text{p}} \leq p_{\text{m}}^*$.
\end{proof}

\begin{proposition}
\label{prop_bounded_mass_gradient}
If $r = s$ in \eqref{gradient_coordinate_constraints},
then $\mathrm{mass}(\nu) \leq C+b-a$ for any $\nu$ feasible in (\ref{problem_gradient_lp_rational_polynomial}).
\end{proposition}

\begin{proof}
 Using the substitution in measures from the proof of Proposition \ref{prop_substitution_measures_equivalence}, see that we have that \an{$C \geq \int_{\Omega_{\text{h}}}|z|^rd\nu = \int_{\Omega_{\text{h}}}(1 - w^r)d\nu = \mathrm{mass}(\nu) - \int_{\Omega_{\text{h}}}1d\mu = \mathrm{mass}(\nu) + a - b$,}
where $\mu$ is a measure feasible for \eqref{problem_gradient_lp_rational} or \eqref{problem_gradient_lp_rational_denominator}. The last equality holds, because $\mu(\Omega_{\text{h}}) = b - a$.
\end{proof}

The mass of any feasible measure $\nu$ being bounded is a crucial property, which makes LP (\ref{problem_gradient_lp_rational_polynomial}) amenable to approximation with the moment-SOS hierarchy, with convergence guarantees, as described \an{in \cite{tacchi_2022}}. 

\begin{remark}
\label{remark_convergence}
    \an{According to \cite[Corollary 8]{tacchi_2022},} the value of the moment-SOS relaxations converge to $p_{\text{p}}^* = p_{\text{r}}^*$ as the relaxation order goes to $\infty$, under standard Archimedean assumptions on the descriptions of the semialgebraic sets involved. \an{This statement relies on} the density of polynomials on compact sets, boundedness of the mass of $\nu$ (Proposition~\ref{prop_bounded_mass_gradient}) and Putinar's Positivstellensatz~\cite{putinar_positive_1993}.
\end{remark}

\section{Examples}\label{sec:examples}

In this section, we describe two examples. The first one demonstrates the use of homogenization in the setting from which we derived the theory. The second one sketches the possible approach in the presence of non-polynomial data. In both examples, we use \textit{GloptiPoly} \cite{henrion_gloptipoly_2009} to model the linear problem on moments and \textit{Mosek} \cite{mosek} to solve the convex relaxations.

\subsection{The Lavrentiev phenomenon}
This famous example is used to illustrate and justify the regularity assumptions made to formulate existence results in calculus of variation, see e.g. \cite[Section 4.7]{Dacorogna_2008}:
\begin{align}
\label{lavr_orig}
    &\inf_{x, u} \int_0^1{\left(t - x^3(t)\right)^2u^6(t)}dt,\\
    \text{s.t.: } & x(0) = 0, \quad x(1) = 1,\notag \\ &\dot{x}(t) = u(t), \quad x(t)\in [-1, 1] \quad \text{a.e. } t\in [0, 1], \notag\\
    &u \in L^6([0, 1]) . \notag
\end{align}
Its optimal solution $x^*(t) = \sqrt[3]{t}$ yields the optimal value 0. The solution belongs to $W^{1, 1}([0, 1]; [-1, 1])$, but it does not belong to $W^{1, p}([0, 1]; [-1, 1])$ for $p \geq \frac{3}{2}$. Note in particular that the solution does not belong to $W^{1, r}([0, 1]; [-1, 1])$ for $r=6$ equal to the degree of the integrand in $u(t)$, which violates the requirement that $u \in L^r([a, b])$. If we proceeded to use the homogenization technique, we would arrive at a reformulation of an LP on measures, for which the mass of the optimal measure is not finite, and thus, this problem would not be approximately solvable with the moment-SOS hierarchy.

Consider instead the following modified problem:
\begin{align}
\label{lavr_modified}
    &\inf_{x, u} \int_0^1{\left(t - x^3\left(t\right)\right)^2 u(t)}dt,\\
    \text{s.t. } & x(0) = 0, \quad x(1) = 1,\notag\\&\dot{x}(t) = u(t) \geq 0, \quad x(t)\in [-1, 1] \quad  \text{a.e. } t\in [0, 1],\notag\\
    &u \in L^1([0, 1]).\notag
\end{align}
This problem has the same optimal solution as the original problem \eqref{lavr_orig}. However, now it holds that $x\in W^{1, r}([0, 1]; [-1, 1])$ with $r=1$. Hence, $u \in L^1([a, b])$, and the homogenization technique can be used. The LP on positive measures corresponding to the reformulation \eqref{problem_ocm} is
\begin{align*}
     &\inf_{\mu \in \mathcal{M}_+\left( \Omega \right)} \int_\Omega{\left(t - x^3\right)^2 u}d\mu,\\
    \text{s.t.: } &\frac{\partial \mu}{\partial t} + \frac{\partial \mu}{\partial x}u + \delta_1\delta_{1} = \delta_0\delta_{0}, \quad \int_{\Omega}ud\mu \leq C,
\end{align*}
where $\Omega = [0, 1]\times [-1, 1] \times \an{[0, \infty)}$ (note \an{that} it is unbounded) and since $u \in L^1([a, b])$ is positive, we can find a constant $C$ such that $\int_a^b |u(t)| dt \leq C$. The constant $C$ is in general hard to find. In this example, the integrand is not coercive (see Remark \ref{remark_coercivity}). However, since we know the optimal solution, we know that it is sufficient to take $C \geq 1$; e.g., let $C = 5$. Notice that to recover polynomial property, we got rid of the absolute value in the moment constraint. Although $r = 1$ is odd, we are able to do that because $u \geq 0$. Following the developments in Sections \ref{sec:lpc} - \ref{sec:lpp}, we are able to arrive at the final reformulation  \eqref{problem_gradient_lp_rational_polynomial} on compact set with polynomial data as follows:
\begin{align*}
    &\inf_{\nu \in \mathcal{M}_+\left( \Omega_{\text{h}} \right)} \int_{\Omega_{\text{h}}}{\left(t - x^3\right)^2 z}d\nu,\\
    \text{s.t.} & \frac{\partial \nu}{\partial t}w + \frac{\partial \nu}{\partial x}z + \delta_1\delta_1 = \delta_0\delta_0, \quad \int_{\Omega_{\text{h}}} z d\nu \leq C,
\end{align*}
where
\begin{align*}
    \Omega_{\text{h}} = \Set{(t, x, z, w) \in \R^4}{\begin{aligned}
  & t \in [0, 1], x \in [-1, 1],\\
  &  (z, w)\in \mathscr B_1, z \geq 0\end{aligned}}.
\end{align*}
According to Remark \ref{remark_convergence}, we can solve it approximately with the moment-SOS hierarchy with convergence guarantees. 

Note that this example is a follow-up to Remark \ref{remark_s_odd}, because the integer $s$ from compactification \eqref{gradient_coordinate_constraints} is odd. Note that due to the additional constraint $z\geq 0$, the set $\Omega_{\text{h}}$ is compact and furthermore, the coordinate change is a bijection despite the fact that $s$ is odd. Therefore, for this example, the usage of $s$ odd is justified.

Note also that while originally $u(t)$ was unconstrained, here  $u(t)$ must be non-negative. However, this does not invalidate the use of homogenization - we just need to express the restriction on $u(t)$ as a restriction on $(z, w)$, which translates into the additional requirement $z \geq 0$ in this situation.

Numerically, we obtain the correct value 0 at 5 significant digits when solving the lowest (i.e. degree 8) relaxation of the moment-SOS hierarchy.

\subsection{The Brachistochrone}
For this textbook example, see e.g. \cite[Example 4.16]{Dacorogna_2008} or \cite[Section 2.1.4]{Liberzon_2012}, we show how the homogenization technique can also be used in a non-polynomial setting. The brachistochrone problem has the following formulation:
\begin{align}
\label{brachistochrone_original}
    &\inf_{x, u}\int_0^1{\sqrt{\frac{1+u^2(t)}{x(t)}}}dt,\\
    \text{s.t. }&x(0) = 0, \quad x(1) = 1, \notag \\
    &\dot{x}(t)=u(t), \quad x(t)\in [0, 1] \quad \text{a.e. } t\in [0, 1], \notag \\
    &u \in L^1([0, 1]) \notag
\end{align}
and its analytical optimal value is equal to $2.5819$ to 5 significant digits. Note that the control variable $u(t)$ can be unbounded. However, the data are not polynomial. The primal LP on measures \eqref{problem_ocm} reads
\begin{align}
\label{brachistochrone_sqrt}
    &\inf_{\mu \in \mathcal M_+(\Omega)}\int_{\Omega}{\sqrt{\frac{1+u^2}{x}}}d\mu,\\
    \text{s.t. }& \frac{\partial \mu}{\partial t} + \frac{\partial \mu}{\partial x}u + \delta_1\delta_1 = \delta_0\delta_0, \notag
\end{align}
where $\Omega = [0,1]\times [0, 1] \times \R$. 
As it was said at the beginning of this section, this example sketches a slightly different approach, since the integrand is not polynomial. For this reason, we need not use the explicit bound on the moment (i.e., the integral $\int_{\Omega}|u|^1d\mu$). This will be later on replaced by a different condition. Following the developments in Section \ref{sec:lpc}, let us simultaneously homogenize the problem and reformulate to obtain polynomial data (get rid of the square root) by a slightly different coordinate change \an{than} what is described in Section \ref{sec:lpc}. Consider \an{the substitution $y = \sqrt{x}$},
which is indeed a bijection, which maps the interval $[0, 1]$ to itself. We see that the boundary values are not changed, i.e., $y(0) = 0$, $y(1) = 1$. We now compute the derivative
\begin{equation*}
    \dot{y} = \frac{1}{2}x^{-\frac{1}{2}}\dot{x} = \frac{u}{2y}.
\end{equation*}
The new variable $y$ is still bounded (as is the $x$ variable), while its derivative may be unbounded \an{(same as the derivative of $x$)}. Hence, we now homogenize and introduce homogenized variables $(z, w)$ by following the developments of Section \ref{sec:lpc}:
\begin{align*}
    z &= wu = 2wy\dot{y},\\
    1 &= z^2 + w^2,\quad w\geq 0.
\end{align*}
We have chosen to homogenize to the compact set $\mathscr B_2$ as in \eqref{gradient_coordinate_constraints} to deal with the square root in the objective functional. Now, if we express the problem (\ref{brachistochrone_sqrt}) in variables $(t, y, z, w)$, we get
\begin{align}
\label{brachistochrone_rational}
    &\inf_{\mu \in \mathcal M_+(\Omega_{\text{h}})}\int_{\Omega_{\text{h}}}{\frac{1}{wy}}d\mu,\\
    \text{s.t. }& \frac{\partial \mu}{\partial t} + \frac{\partial \mu}{\partial x}\frac{z}{2wy} + \delta_1\delta_1 = \delta_0\delta_0, \notag
\end{align}
where
\begin{equation*}
    \Omega_{\text{h}} = \Set{(t, y, z, w) \in \R^4}{\begin{aligned}&t \in [0, 1], y \in [0, 1],\\ &(z, w) \in \mathscr B_2\end{aligned}},
\end{equation*}
which corresponds to LP \eqref{problem_gradient_lp_rational} in the described method.

What remains to be done now is to reformulate the problem (\ref{brachistochrone_rational}) with polynomial data. We do it \an{similarly} to section \ref{recover_polynomial_property}. That is, we use the substitution in measures $d\nu = \frac{1}{wy}d\mu$, as it is described in the proof of Proposition \ref{prop_substitution_measures_equivalence}. Note that $wy$ is a measurable function, which justifies the substitution in measures. The difference in the substitution (compared to what was used in Section \ref{sec:lpp}) is caused by the original data not being polynomial. We obtain the following LP
\begin{align*}
    &\inf_{\nu \in \mathcal M_+(\Omega_{\text{h}})}\int_{\Omega_{\text{h}}}d\nu,\\
    \text{s.t.: }& \frac{\partial \nu}{\partial t}wy + \frac{\partial \nu}{\partial x}\frac{z}{2} + \delta_1\delta_1 = \delta_0\delta_0, \quad \int_{\Omega_{\text{h}}}d\nu \leq C,
\end{align*}
 where the additional inequality is used to preserve finite mass of measure $\nu$. Since the mass of $\nu$ is the objective functional, we can bound it by the value of the objective functional of \eqref{brachistochrone_original} for any admissible trajectory $x(t)$. We can choose $x(t) = t$, for which $C=\int_0^1 \sqrt{\frac{2}{t}}dt = 2\sqrt{2}$. Thus, according to Remark \ref{remark_convergence}, we can solve this problem with the moment-SOS hierarchy with sound convergence guarantees. Indeed, numerically, we observe that it achieves the lower bounds $2.0000$, $2.5578$ and $2.5819$ (the global minimum) at 5 significant digits 
 for relaxation degrees $2$, $4$ and $6$.

\section{Conclusion}\label{sec:conclusion}

In this contribution, we described how the moment-SOS hierarchy can be used to solve optimal control problems with polynomial data and unbounded controls. We used a series of linear reformulations and finally reached a reformulation on compact sets, which can be solved by the hierarchy with convergence guarantees. Moreover, we proved that all these reformulations are equivalent to the first reformulation using relaxed controls, i.e. there is no relaxation gap, their optimal values are the same.


For simplicity of the exposition, we considered the particular case of scalar (i.e. single input) calculus of variations problems (i.e. the control is the velocity). Our approach is however, not limited to scalar problems, and it can be extended to multivariate problems with controlled vector fields, as well. We are also investigating the applications of this homogenization technique to unbounded states, as well as to optimal control problems of partial differential equations.

\section{Acknowledgements}
This work benefited from discussions with Rodolfo R\'ios-Zertuche, pointing out the work of Patrick Bernard. 
 

\end{document}